\documentclass{amsart}

\usepackage{latexsym}
\usepackage{amssymb, latexsym}
\usepackage{amsmath}
\usepackage{mathrsfs}
\usepackage{amsxtra}
\newtheorem{theorem}{Theorem}[section]
\newtheorem{corollary}[theorem]{Corollary}

\newtheorem{lemma}[theorem]{Lemma}

\newtheorem{prop}[theorem]{Proposition}

\theoremstyle{definition}
\newtheorem{definition}[theorem]{Definition}
\newtheorem{question}[theorem]{Question}

\newcommand{\tail}{\text{tail}}
\newcommand{\x}{\mathfrak{X}}
\newcommand{\y}{\mathfrak{Y}}
\newcommand{\n}{\mathbb {N}}
\newcommand{\p}{\mathbb{P}}
\newcommand{\q}{\mathbb{Q}}

\newcommand{\fin}{\mathrm{Fin}}
\newcommand{\power}{\mathcal{P}}

\newcommand{\la}{\langle}
\newcommand{\ra}{\rangle}
\begin{document}
\author{Victoria Gitman}
\today
\address{New York City College of Technology (CUNY),
Mathematics, 300 Jay Street, Brooklyn, NY 11201 USA}
\email{vgitman@nylogic.org}
\title{Proper and Piecewise Proper Families of Reals}
\maketitle
\begin{abstract}
I introduced the notions of proper and piecewise proper families of
reals to make progress on a long standing open question in the field
of models of Peano Arithmetic about whether every Scott set is the
standard system of a model of {\rm PA}. A Scott set is a family of
reals closed under $\Delta_1$ definability and satisfying weak
Konig's Lemma. A family of reals $\x$ is proper if it is
arithmetically closed and the quotient Boolean algebra $\x/\fin$ is
a proper partial order. A family is piecewise proper if it is the
union of a chain of proper families of size $\leq\omega_1$. I showed
that under the Proper Forcing Axiom, every proper or piecewise
proper family of reals is the standard system of a model of PA.
Here, I explore the question of the existence of proper and
piecewise proper families of reals of different cardinalities.
\end{abstract}
\section{Introduction}
One of the central concepts in the field of models of Peano
Arithmetic is the \emph{standard system} of a model of {\rm PA}. The
\emph{standard system} of a model of {\rm PA} is the collection of
subsets of the natural numbers that arise as intersections of the
definable sets of the model with its \emph{standard part} $\n$. The
notion of a \emph{Scott set} captures three key properties of
standard systems.
\begin{definition}
$\x\subseteq \power(\n)$ is a \emph{Scott set} if
\begin{itemize}
\item [(1)] $\x$ is a Boolean algebra of sets.
\item [(2)] If $A\in\x$ and $B$ is Turing computable from
$A$, then $B\in\x$.
\item [(3)] If $T$ is an infinite binary tree coded by a set in
$\x$, then $\x$ has a set coding some path through $T$.
\end{itemize}
\end{definition}
In 1962, Scott showed that every standard system is a Scott set and
the partial converse that every \emph{countable} Scott set is the
standard system of a model of {\rm PA} \cite{scott:ssy}. The
question of whether \emph{every} Scott is the standard system of a
model of {\rm PA} became known in the folklore as Scott's Problem.
In 1982, Knight and Nadel extended Scott's result to Scott sets of
size $\omega_1$ \cite{knight:scott}. They showed that every Scott
set of size $\omega_1$ is the standard system of a model of {\rm
PA}. It has proved very difficult to make further progress on
Scott's Problem. My approach, following Engstr\"om
\cite{engstrom:thesis} and suggested several years earlier by
Hamkins, Marker, etc., has been to use the set theoretic techniques
of forcing and the forcing axioms.

We can associate with every family of reals $\x$, the poset
$\x/\fin$ which consists of the infinite sets of $\x$ under the
ordering of \emph{almost inclusion}. Engstr\"om in
\cite{engstrom:thesis} introduced the use of this poset in
connection with Scott's Problem. A family of reals is
\emph{arithmetically closed} if whenever $A$ is in it and $B$ is
arithmetically definable from $A$, then $B$ is also in it. The
arithmetic closure of $\x$ is an essential ingredient in the
constructions that make the posets $\x/\fin$ useful in investigating
properties of uncountable models of {\rm PA} (for details of the
constructions, see \cite{gitman:scott}). For this reason, whenever
we view a family of reals as a poset we will always assume
arithmetic closure. A family $\x$ is \emph{proper} if it is
arithmetically closed and the poset $\x/\fin$ is proper. A family
$\x$ is piecewise proper if it is the union of a chain of proper
families each of which has size $\leq\omega_1$. I showed in
\cite{gitman:scott} that under the Proper Forcing Axiom ({\rm PFA}),
every proper or piecewise proper family of reals is the standard
system of a model of {\rm PA}. I will give an extended discussion of
properness and the {\rm PFA} in Section \ref{sec:proper}.

Throughout the paper, I equate \emph{reals} with subsets of $\n$. It
is easy to see that every countable arithmetically closed family of
reals is proper and $\power(\n)$ is proper as well (see Section
\ref{sec:forcing}). Every arithmetically closed family of size
$\leq\omega_1$ is trivially piecewise proper since it is the union
of a chain of countable arithmetically closed families. It becomes
much more difficult to find instances of uncountable proper families
of reals other than $\power(\n)$. Also, it was not clear for a while
whether there are piecewise proper families of of size larger than
$\omega_1$. My main results are:

\begin{theorem}
If {\rm CH} holds, then $\power^V(\n)/\fin$ remains proper in any
generic extension by a c.c.c.\ poset.
\end{theorem}

\begin{theorem}
There is a generic extension of $V$ by a c.c.c.\ poset, which
contains continuum many proper families of reals of size $\omega_1$.
\end{theorem}

\begin{theorem}
There is a generic extension of $V$ by a c.c.c.\ poset, which
contains continuum many piecewise proper families of reals of size
$\omega_2$.
\end{theorem}

\section{Proper Posets and the PFA}\label{sec:proper}
Proper posets were invented by Shelah, who sought a class of
$\omega_1$ preserving posets that would extend the c.c.c.\ and
countably closed classes of posets and be preserved under iterations
with countable support. The Proper Forcing Axiom ({\rm PFA}) was
introduced by Baumgartner who showed that it is consistent by
assuming the existence of a supercompact cardinal
\cite{baumgartner:pfa}. This remains the best known upper bound on
the consistency of {\rm PFA}.

Recall that for a cardinal $\lambda$, the set $H_\lambda$ is the
collection of all sets whose transitive closure has size less than
$\lambda$. Let $\p$ be a poset and $\lambda$ be a cardinal greater
than $2^{|\p|}$. Since we can always take an isomorphic copy of $\p$
on the cardinal $|\p|$, we can assume without loss of generality
that $\p$ and $\power(\p)$ are elements of $H_\lambda$. In
particular, we want to ensure that all dense subsets of $\p$ are in
$H_\lambda$. Let $M$ be a countable elementary submodel of
$H_\lambda$ containing $\p$ as an element. If $G$ is a filter on
$\p$, we say that $G$ is $M$-\emph{generic} if for every maximal
antichain $A\in M$ of $\p$, the intersection $G\cap A\cap
M\neq\emptyset$. It must be explicitly specified what $M$-generic
means in this context since the usual notion of generic filters
makes sense only for transitive structures and $M$ is not
necessarily transitive. This definition of $M$-generic is closely
related to the definition for transitive structures. To see this,
let $M^*$ be the Mostowski collapse of $M$ and $\p^*$ be the image
of $\p$ under the collapse. Let $G^*\subseteq \p^*$ be the pointwise
image of $G\cap M$ under the collapse. Then $G$ is $M$-generic if
and only if $G^*$ is $M^*$-generic for $\p^*$ in the usual sense.

Later we will need the following important characterization of
$M$-generic filters.

\begin{theorem} If\/ $\p$ is a poset in $M\prec H_\lambda$, then a
$V$-generic filter $G\subseteq\p$ is $M$-generic if and only if
$M\cap \text{\emph{Ord}}=M[G]\cap\text{\emph{Ord}}$. \emph{(See, for
example, \cite{shelah:proper}, p. 105)}\label{th:equproper}
\end{theorem}

\begin{definition} Let $\p\in H_\lambda$ be a poset and $M$ be an
elementary submodel of $H_\lambda$ containing $\p$. Then a condition
$q\in\p$ is $M$-\emph{generic} if and only if every $V$-generic
filter $G\subseteq\p$ containing $q$ is $M$-generic.
\end{definition}
\begin{definition} A poset $\p$ is \emph{proper} if for every
$\lambda> 2^{|\p|}$ and every countable $M\prec H_\lambda$
containing $\p$, for every $p\in\p\cap M$, there is an $M$-generic
condition below $p$.
\end{definition}

When proving that a poset is proper it is often easier to use the
following equivalent characterization which appears in
\cite{shelah:proper} (p.\ 102).
\begin{theorem}\label{th:equivproper}
A poset $\p$ is proper if there exists a $\lambda>2^{|\p|}$ and a
club of countable $M\prec H_\lambda$ containing $\p$, such that for
every $p\in\p\cap M$, there is an $M$-generic condition below $p$.
\end{theorem}

Countably closed posets and c.c.c.\ posets are proper and all proper
posets preserve $\omega_1$ \cite{shelah:proper}.

\begin{definition}
\emph{The Proper Forcing Axiom} ({\rm PFA}) is the assertion that
for every proper poset $\p$ and every collection $\mathcal D$ of at
most $\omega_1$ many dense subsets of $\p$, there is a filter on
$\p$ that meets all of them.
\end{definition}

The Proper Forcing Axiom decides the size of the continuum. It was
shown in \cite{veli:pfa} that under {\rm PFA}, continuum is
$\omega_2$.

\section{Proper and Piecewise Proper Families}\label{sec:forcing}

Let $\x$ be a family of reals. Define the poset $\x/\fin$ to consist
of the infinite sets in $\x$ under the ordering of almost inclusion.
That is, for infinite $A$ and $B$ in $\x$, we say that $A\leq B$ if
and only if $A\subseteq_\fin B$. Observe that $\x/\fin$ is forcing
equivalent to forcing with the Boolean algebra $\x$ modulo the ideal
of finite sets. A familiar and thoroughly studied instance of this
poset is $\power(\omega)/\fin$.  For a property of posets $\mathscr
P$, if $\x$ is an arithmetically closed family of reals and
$\x/\fin$ has $\mathscr P$, I will simply say that \emph{$\x$ has
property $\mathscr P$}. An important point to be noted here is that
whenever a family $\x$ is discussed as a poset, I will always be
assuming that \emph{it is arithmetically closed}. Recall that the
reason for this is the need for arithmetic closure of $\x$ in the
constructions with models of {\rm PA} in which $\x/\fin$ is used.

The easiest way to show that a poset is proper to show that it is
c.c.c.\ or countably closed. Thus, every countable arithmetically
closed family is proper since it is c.c.c.\ and $\power(\n)$ is
proper since it is countably closed. Unfortunately these two
conditions do not give us any other instances of proper families.

\begin{theorem}\label{th:wrong} Every c.c.c.\ family of reals is countable.
\end{theorem}
\begin{proof}
Let $\x$ be an arithmetically closed family of reals.  If $x$ is a
finite subset of $\n$, let $\ulcorner x \urcorner$ denote the code
of $x$ using G\"{o}del's coding. For every $A\in\x$, define an
associated $A'=\{\ulcorner A\cap n\urcorner\mid n\in\n\}$. Clearly
$A'$ is definable from $A$, and hence in $\x$. Observe that if
$A\neq B$, then $|A'\cap B'|<\omega$. Hence if $A\neq B$, we get
that $A'$ and $B'$ are incompatible in $\x/\fin$. It follows that
$\mathscr A=\{A'\mid A\in\x\}$ is an antichain of $\x/\fin$ of size
$|\x|$.  This shows that $\x/\fin$ always has antichains as large as
the whole poset.
\end{proof}

Thus, the poset $\x/\fin$ has the worst possible chain condition,
namely $|\x|^+$-c.c..

\begin{theorem} \label{th:countclosed} Every countably closed family
of reals is $\power(\n)$.
\end{theorem}
\begin{proof}
I will show that every $A\subseteq \n$ is in $\x$. Define a sequence
of subsets $\la B_n\mid n\in\omega\ra$ by $B_n=\{m\in\n\mid
(m)_n=\chi_A(n)\}$ where $\chi_A$ is the characteristic function of
$A$. Let $A_m=\cap_{n\leq m}B_n$ and observe that each $A_m$ is
infinite. Thus, $A_0\geq A_1\geq\cdots \geq A_m\geq\ldots$ are
elements of $\x/\fin$. By countable closure, there exists $C\in
\x/\fin$ such that $C\subseteq_{\fin}A_m$ for all $m\in \n$. Thus,
$C\subseteq_\fin B_n$ for all $n\in \n$. It follows that
$A=\{n\in\n\mid\exists m\,\forall k\in C\,\text{ if }k>m, \text{
then } (k)_n=1\}$. This shows that $A$ is arithmetic in $C$, and
hence $A\in\x$ by arithmetic closure. Since $A$ was arbitrary, this
concludes the proof that $\x=\power(\n)$.
\end{proof}
The assumption that $\x$ is arithmetically closed is not necessary
for Theorem \ref{th:countclosed}. Any family of reals $\x$ such that
$\x/\fin$ is countably closed must be arithmetically closed (see
\cite{gitman:scott}) .

Enayat showed in \cite{enayat:endextensions} that {\rm ZFC} proves
the existence of an arithmetically closed family of size $\omega_1$
which collapses $\omega_1$, and hence is \emph{not} proper.

Later Enayat and Shelah showed in \cite{shelah:borel} that there is
a \emph{Borel} arithmetically closed family of size $\omega_1$ which
is \emph{not} proper as well.

I will show below that it is consistent with {\rm ZFC} that there
are continuum many proper families of size $\omega_1$ and it is
consistent with {\rm ZFC} that there are continuum many piecewise
proper families of size $\omega_2$. But first I will consider the
question of when does forcing to add new reals preserve the
properness of the reals of the ground model. I will show that if
{\rm CH} holds, forcing with a c.c.c.\ poset preserves the
properness of the reals of the ground model.

\begin{lemma}\label{le:intersect}
Let $\x_0\subseteq\x_1\subseteq\cdots\subseteq
\x_\xi\subseteq\cdots$ for $\xi<\omega_1$ be a continuous chain of
countable families of reals and let $\x=\cup_{\xi<\omega_1} \x_\xi$.
If $M$ is a countable elementary substructure of some $H_\lambda$
and $\la \x_\xi\mid\xi<\omega_1\ra\in M$, then $M\cap \x=\x_\alpha$
where $\alpha=\text{\emph{Ord}}^M\cap\omega_1$.
\end{lemma}
\begin{proof}
Let $\alpha=\text{Ord}^M\cap\omega_1$. Suppose $\xi\in\alpha$, then
$\xi\in M$, and hence $\x_\xi\in M$. Since $\x_\xi$ is countable, it
follows that $\x_\xi\subseteq M$. Thus, $\x_\alpha\subseteq\x\cap
M$. Now suppose $A\in \x\cap M$, then the least $\xi$ such that
$A\in\x_\xi$ is definable in $H_\lambda$. It follows that $\xi\in
M$, and hence $\xi\in\alpha$. Thus, $\x\cap M\subseteq \x_\alpha$.
\end{proof}
\begin{lemma}\label{le:count}
 Suppose $\p$ is a c.c.c.\ poset and $G\subseteq \q$ is $V$-generic
for a countably closed poset $\q$. Then  $\p$ remains c.c.c.\ in
$V[G]$.
\end{lemma}
\begin{proof}
Suppose $\p$ does not remain c.c.c.\ in $V[G]$.  Fix a $\q$-name
$\dot{A}$ and $r\in \q$ such that $r\Vdash ``\dot{A}$ is a maximal
antichain of $\check{\p}\text{ of size }\omega_1"$. Choose $q_0\leq
r$ and $a_0\in \p$ such that $q_0\Vdash \check{a_0}\in \dot{A}$.
Suppose that we have defined $q_0\geq q_1\geq\cdots\geq
q_\xi\geq\cdots$ for $\xi<\beta$ where $\beta$ is some countable
ordinal, together with a corresponding sequence $\langle a_\xi\mid
\xi<\beta\rangle$ of elements of $\p$ such that $q_\xi\Vdash
\check{a}_\xi\in \dot{A}$ and $a_{\xi_1}\neq a_{\xi_2}$ for all
$\xi_1<\xi_2$. By countable closure of $\q$, we can find $q\in\q$
such that $q\leq q_\xi$ for all $\xi<\beta$.  Let $q_\beta\leq q$
and $a_\beta\in \p$ such that $q_\beta\Vdash \check{a}_\beta\in
\dot{A}$ and $a_\beta\neq a_\xi$ for all $\xi<\beta$.  Such
$a_\beta$ must exist since we assumed $r\Vdash ``\dot{A} \text{ is a
maximal antichain of }\check{\q}\text{ of size }\omega_1"$ and
$q\leq r$. Thus, we can build a descending sequence $\langle
q_\xi\mid \xi<\omega_1\rangle$ of elements of $\q$ and a
corresponding sequence $\langle a_\xi\mid \xi<\omega_1\rangle$ of
elements of $\p$ such that $q_\xi\Vdash \check{a}_\xi\in\dot{A}$.
But clearly $\langle a_\xi\mid \xi<\omega_1\rangle$ is an antichain
in $V$ of size $\omega_1$, which contradicts the assumption that
$\p$ was c.c.c..
\end{proof}
\begin{theorem}\label{th:oldreals}
If {\rm CH} holds, then $\power^V(\n)/\fin$ remains proper in any
generic extension by a c.c.c.\ poset.
\end{theorem}
\begin{proof}
Let $\p$ be a c.c.c.\ poset and fix a $V$-generic $g\subseteq \p$.
In $V$, let $\power(\n)=\cup_{\xi<\omega_1} \x_\xi$ where each
$\x_\xi$ is countable and
$\x_0\subseteq\x_1\subseteq\cdots\subseteq\x_\xi\subseteq\cdots$ is
a continuous chain. For sufficiently large cardinals $\lambda$, it
is easy to see that $H_\lambda^{V[g]}=H_\lambda[g]$. The countable
elementary substructures of $H_\lambda[g]$ of the form $M[g]$ where
$M\subseteq V$ and $\la \x_\xi\mid \xi<\omega_1\ra, \p\in M[g]$ form
a club. So by Theorem \ref{th:equivproper}, it suffices to find
generic conditions only for such elementary substructures. Fix a
countable $M[g]\prec H_\lambda[g]$ in $V[g]$ such that $\la
\x_\xi\mid\xi<\omega_1\ra, \p\in M[g]$ and $M\subseteq V$. We need
to prove that for every $B\in M[g]\cap \power(\n)^V/\fin$, there
exists $A\in \power(\n)^V/\fin$ such that $A\subseteq_\fin B$ and
$A$ is $M[g]$-generic in $V[g]$. By Lemma \ref{le:intersect},
$M[g]\cap \power^V(\n)=\x_\alpha$ where $\alpha=\text{Ord}^M\cap
\omega_1$. Let $\mathcal D=\{\mathscr D\cap \x_\alpha\mid \mathscr
D\in M\text{ and }\mathscr D\text{ dense in }\power^V(\n)/\fin\}$.
Observe that $\mathcal D\subseteq V$ and $|\mathcal D|=\omega$.
Since $\p$ is c.c.c., we can show that there is $\mathcal
D'\supseteq \mathcal D$ of size $\omega$ in $V$. In $V$, use
$\mathcal D'$ and $\x_\alpha$ to define $\mathcal E=\{\mathscr D\in
\mathcal D'\mid \mathscr D\text{ dense in }\x_\alpha\}$. It is clear
that $\mathcal D\subseteq \mathcal E$. By the countable closure of
$\power(\n)^V/\fin$ in $V$, we can find an infinite $A\subseteq_\fin
B$ such that every $\mathscr D\in \mathcal E$ contains some $C$
above $A$. It follows that $A$ is $M$-generic. In fact, I will show
that $A$ is $M[g]$-generic. To verify this, we need to check that
whenever $A\in G$ and $G\subseteq \power(\n)^V/\fin$ is
$V[g]$-generic, then $M[g][G]\cap \text{Ord}=M[g]\cap \text{Ord}$.
Since we are forcing with $\p\times\power(\n)/\fin$, we have
$M[g][G]=M[G][g]$. It is clear that $M\cap \text{Ord}=M[g]\cap
\text{Ord}$, and so it remains to show that $M[G][g]\cap
\text{Ord}=M\cap \text{Ord}$. Since $A\in G$ and $A$ is $M$-generic,
we have that $M[G]\cap \text{Ord}=M\cap \text{Ord}$. The poset $\p$
remains c.c.c.\ in $V[G]$ by Lemma \ref{le:count} since
$\power(\n)^V/\fin$ is countably closed. Also we have $M[G]\prec
H_\lambda[G]$, even though $M[G]$ itself may not be an element of
$V[G]$. Let $\mathscr A$ be a maximal antichain of $\p$ in $M[G]$,
then $\mathscr A\in H_\lambda[G]$, and hence $\mathscr A$ has size
$\omega$. It follows that $\mathscr A\subseteq M[G]$. Since $g$ is
$V[G]$-generic, it must meet $\mathscr A$. So $g$ is $M[G]$-generic,
and hence $M[G][g]\cap \text{Ord}=M[G]\cap \text{Ord}$.
\end{proof}
It follows that it is consistent that there are uncountable proper
families other than $\power(\n)$. Start in any universe satisfying
{\rm CH} and force to add a Cohen real. In the resulting generic
extension, the reals of $V$ will be an uncountable proper family.

Next, I will show how to force the existence of many proper
families. I will begin by looking at what properness translates into
in this specific context.
\begin{prop}\label{prop:proper} Suppose $\x$ is a family of reals and
$\mathscr A$ is a countable antichain of $\x/\fin$.  Then  for $B\in
\x$:
\begin{itemize}
\item[(1)] Every $V$-generic filter $G\subseteq \x/\fin$ containing $B$ meets
$\mathscr A$.
\item[(2)] There exists a finite list $A_0,\ldots, A_n\in \mathscr A$
such that $B\subseteq_\fin A_0\cup\ldots\cup A_n$.
\end{itemize}
\end{prop}
\begin{proof}$\,$\\
(2)$\Longrightarrow$(1): Suppose $B\subseteq_\fin A_0\cup\ldots\cup
A_n$ for some $A_0,\ldots, A_n\in \mathscr A$. Since a $V$-generic
filter $G$ is an ultrafilter, one of the $A_i$ must be in $G$.\\
(1)$\Longrightarrow$(2): Assume that every $V$-generic filter $G$
containing $B$ meets $\mathscr A$ and suppose toward a contradiction
that $(2)$ does not hold. Enumerate $\mathscr A=\{A_0,A_1,\ldots,
A_n,\ldots\}$. It follows that for all $n\in\n$, the intersection
$B\cap (\n-A_0)\cap\cdots\cap(\n-A_n)$ is infinite. Define
$C=\{c_n\mid n\in\n\}$ such that $c_0$ is the least element of
$B\cap(\n-A_0)$ and $c_{n+1}$ is the least element of
$B\cap(\n-A_0)\cap\cdots\cap (\n-A_{n+1})$ greater than $c_n$.
Clearly $C\subseteq B$ and $C\subseteq_\fin (\n-A_n)$ for all $n\in
\n$.  Let $G$ be a $V$-generic filter containing $C$, then $B\in G$
and $(\n-A_n)\in G$ for all $n\in\n$. But this contradicts our
assumption that $G$ meets $\mathscr A$.
\end{proof}
\begin{corollary}\label{cor:properequiv}
A family of reals $\x$ is proper if and only if there exists
$\lambda>2^{|\x|}$ such that for every countable $M\prec H_\lambda$
containing $\x$, whenever $C\in M\cap \x/\fin$, then there is
$B\subseteq_\fin C$ in $\x/\fin$ such that for every maximal
antichain $\mathscr A\in M$ of $\x/\fin$, there are $A_0,\ldots,
A_n\in\mathscr A\cap M$ with $B\subseteq_\fin A_0\cup\cdots\cup
A_n$.
\end{corollary}
\begin{proof}$\,$\\
($\Longrightarrow$): Suppose $\x$ is proper. Then there is
$\lambda>2^{|\x|}$ such that for every countable $M\prec H_\lambda$
containing $\x$ and every $C\in M\cap \x/\fin$, there is an
$M$-generic $B\subseteq_\fin C$ in $\x/\fin$. Fix a countable
$M\prec H_\lambda$ containing $\x$ and $C\in M\cap \x/\fin$. Let
$B\subseteq_\fin C$ be $M$-generic. Thus, every $V$-generic filter
containing $B$ must meet $\mathscr A\cap M$ for every maximal
antichain $\mathscr A\in M$ of $\x/\fin$. But since $\mathscr A\cap
M$ is countable, by Proposition \ref{prop:proper}, there exist
$A_0,\ldots,A_n\in \mathscr A\cap M$
such that $B\subseteq_\fin A_0\cup\cdots\cup A_n$.\\
($\Longleftarrow$): Suppose that there is $\lambda>2^{|\x|}$ such
that for every countable \hbox{$M\prec H_\lambda$} containing $\x$,
whenever $C\in M\cap \x/\fin$, then there is $B\subseteq_\fin C$ in
$\x/\fin$ such that for every maximal antichain $\mathscr A\in M$ of
$\x/\fin$, there are $A_0,\ldots, A_n\in\mathscr A\cap M$ with
$B\subseteq_\fin A_0\cup\cdots\cup A_n$. Fix a countable $M\prec
H_\lambda$ with $\x\in M$ and $C\in M\cap \x/\fin$. Let
$B\subseteq_\fin C$ be as above. By Proposition \ref{prop:proper},
every $V$-generic filter $G$ containing $B$ must meet $\mathscr
A\cap M$ for every maximal antichain $\mathscr A\in M$. Thus, $B$ is
$M$-generic. Since $M$ was arbitrary, we can conclude that $\x$ is
proper.
\end{proof}
The hypothesis of Corollary \ref{cor:properequiv} can be weakened,
by Theorem \ref{th:equivproper}, to finding for some $H_\lambda$,
only a club of countable $M$ having the desired property.

The next definition is key to all the remaining arguments in the
paper.
\begin{definition}
Let $\x$ be a countable family of reals, let $\mathcal D$ be some
collection of dense subsets of $\x/\fin$, and let $B\in \x$. We say
that an infinite set $A\subseteq\n$ is $\langle \x,\mathcal
D\rangle$\emph{-generic below} $B$ if $A\subseteq_\fin B$ and for
every $\mathscr D\in\mathcal D$, there is $C\in \mathscr D$ such
that $A\subseteq_\fin C$.
\end{definition}

Here one should think of the context of having some large family
$\y\in M\prec H_\lambda$ for a countable $M$, $\x=\y\cap M$, and
$\mathcal D=\{\mathscr D\cap M\mid \mathscr D\in M\text{ and
}\mathscr D\text{ dense in }\y/\fin\}$. We think of $A$ as coming
from the large family $\y$ and the requirement for $A$ to be
$\langle \x,\mathcal D\rangle$-generic is a strengthening of the
requirement to be $M$-generic.
\begin{lemma}\label{f:gen_elt}
Let $\x$ be a countable family.  Assume that $B\in\x/\fin$ and
$G\subseteq \x/\fin$ is a $V$-generic filter containing $B$. Then in
$V[G]$, there is an infinite $A\subseteq\n$ such that
$A\subseteq_\fin C$ for all $C\in G$. Furthermore, if $\mathcal D$
is the collection of dense subsets of $\x/\fin$ of $V$, then such an
$A$ is $\langle \x,\mathcal D\rangle$-generic below $B$.
\end{lemma}
\begin{proof}
Since $G$ is countable and directed in $V[G]$, there exists an
infinite $A\subseteq\n$ such that $A\subseteq_\fin C$ for all $C\in
G$. For the ``furthermore" part, fix a dense subset $\mathscr D$ of
$\x/\fin$ in $V$. Since there is $C\in G\cap \mathscr D$, we have
$A\subseteq_\fin C$. It is clear that $A\subseteq_\fin B$ since
$B\in G$.
\end{proof}
\begin{lemma}\label{le:gen_elt}
Let $\x_0\subseteq\x_1\subseteq\cdots\subseteq
\x_\xi\subseteq\cdots$ for $\xi<\omega_1$ be a continuous chain of
countable families of reals and let $\x=\cup_{\xi<\omega_1} \x_\xi$.
Assume that for every $\xi<\omega_1$, if $B\in \x_\xi$ and $\mathcal
D$ is a countable collection of dense subsets of $\x_\xi$, there is
$A\in \x/\fin$ that is $\langle \x_\xi,\mathcal D\rangle$-generic
below $B$. Then $\x$ is proper.
\end{lemma}
\begin{proof}
Fix a countable $M\prec H_\lambda$ such that $\langle \x_\xi\mid
\xi<\omega_1\rangle\in M$. It suffices to show that generic
conditions exist for such $M$ since these form a club. By Lemma
\ref{le:intersect}, $\x\cap M=\x_\alpha$ where
$\alpha=\text{Ord}^M\cap \omega_1$. Fix $B\in\x_\alpha$ and let
$\mathcal D=\{\mathscr D\cap M\mid \mathscr D\in M \text{ and
}\mathscr D \text{ dense in }\x/\fin\}$. By hypothesis, there is
$A\in\x/\fin$ that is $\langle\x_\alpha,\mathcal D\rangle$-generic
below $B$. Clearly $A$ is $M$-generic.  Thus, we were able to find
an $M$-generic element below every $B\in M\cap \x/\fin$.
\end{proof}
We are finally ready to show how to force the existence of a proper
family of size $\omega_1$.
\begin{theorem}\label{th:forcing}
There is a generic extension of $V$ by a c.c.c.\ poset that
satisfies $\neg{\rm CH}$ and contains a proper family of reals of
size $\omega_1$.
\end{theorem}

\begin{proof}
First, note that we can assume without loss of generality that
$V\models \neg{\rm CH}$ since this is forceable by a c.c.c.\
forcing.

The forcing to add a proper family of reals will be a c.c.c.\ finite
support iteration $\p$ of length $\omega_1$. The iteration $\p$ will
add, step-by-step, a continuous chain
$\x_0\subseteq\x_1\subseteq\cdots\subseteq\x_\xi\subseteq\cdots$ for
$\xi<\omega_1$ of countable arithmetically closed families such that
$\cup_{\xi<\omega_1} \x_\xi$ will have the property of Lemma
\ref{le:gen_elt}. The idea will be to obtain generic elements for
$\x_\xi$, as in Lemma \ref{f:gen_elt}, by adding generic filters.
Once $\x_\xi$ has been constructed, I will force over $\x_\xi/\fin$
below every one of its elements cofinally often before the iteration
is over.  Every time such a forcing is done, I will obtain a generic
element for a new collection of dense sets. This element will be
added to $\x_{\delta+1}$ where $\delta$ is the stage at which the
forcing was done.

Fix a bookkeeping function $f$ mapping $\omega_1$ onto
$\omega_1\times\omega$, having the properties that every pair
$\langle \alpha, n\rangle$ appears cofinally often in the range and
if $f(\xi)=\langle \alpha, n\rangle$, then $\alpha\leq \xi$.  Let
$\x_0$ be any countable arithmetically closed family and fix an
enumeration $\x_0=\{B_0^0,B_1^0,\ldots, B_n^0\ldots\}$. Each
subsequent $\x_\xi$ will be created in $V^{\p_\xi}$. Suppose
$\lambda$ is a limit and $G_\lambda$ is generic for $\p_\lambda$. In
$V[G_\lambda]$, define $\x_\lambda=\cup_{\xi<\lambda} \x_\xi$ and
fix an enumeration
$\x_\lambda=\{B_0^\lambda,B_1^\lambda,\ldots,B_n^\lambda,\ldots\}$.
Consult $f(\lambda)=\langle \xi,n\rangle$ and define
$\dot{\q}_\lambda=\x_\xi/\fin$ below $B_n^\xi$. Suppose
$\delta=\beta+1$, then $\p_\delta=\p_\beta*\dot{\q}_\beta$ where
$\dot{\q}_\beta$ is $\x_\xi/\fin$ for some $\xi\leq\beta$ below one
of its elements.  In $V[G_\delta]=V[G_\beta][H]$, let
$A\subseteq_\fin B$ for all $B\in H$ and define $\x_\delta$ to be
the arithmetic closure of $\x_\beta$ and $A$.  Also in
$V[G_\delta]$, fix an enumeration
$\x_\delta=\{B_0^\delta,B_1^\delta,\ldots,B_n^\delta,\ldots\}$.
Consult $f(\delta)=\langle \xi,n\rangle$ and define
$\dot{\q}_\delta=\x_\xi/\fin$ below $B_n^\xi$. At limits, use finite
support.

The poset $\p$ is c.c.c.\ since it is a finite support iteration of
c.c.c.\ posets (see \cite{jech:settheory}, p.\ 271).  Let $G$ be
$V$-generic for $\p$. It should be clear that we can use $G$ in
$V[G]$ to construct an arithmetically closed Scott set
$\x=\cup_{\xi<\omega_1}\x_\xi$. A standard nice name counting
argument shows that $(2^{\omega})^V=(2^{\omega})^{V[G]}$. Since we
assumed at the beginning that $V\models \neg{\rm CH}$, it follows
that $V[G]\models\neg{\rm CH}$.

Finally, we must see that $\x$ satisfies the hypothesis of Lemma
\ref{le:gen_elt} in $V[G]$. Fix $\x_\xi$, a set $B\in X_\xi$, and a
countable collection $\mathcal D$ of dense subsets of $\x_\xi/\fin$.
Since the poset $\p$ is a finite support c.c.c.\ iteration and all
elements of $\mathcal D$ are countable, they must appear at some
stage $\alpha$ below $\omega_1$. Since we force with $\x_\xi/\fin$
below $B$ cofinally often, we have added a $\langle \x_\xi, \mathcal
D\rangle$-generic condition below $B$ at some stage above $\alpha$.
\end{proof}
\begin{corollary}
There is a generic extension of $V$ that satisfies  {\rm CH}  and
contains a proper family of reals of size $\omega_1$ other than
$\power(\n)$.
\end{corollary}
\begin{proof}
As before, we can assume without loss of generality that $V\models
\neg{\rm CH}$. Force with $\p*\dot{\q}$ where $\p$ is the forcing
iteration from Theorem \ref{th:forcing} and $\q$ is the poset which
adds a subset to $\omega_1$ with countable conditions. Let $G*H$ be
$V$-generic for $\p*\dot{\q}$, then clearly {\rm CH} holds in
$V[G][H]$. Also the family $\x$ created from $G$ remains proper in
$V[G][H]$ since $\q$ is a countably closed forcing, and therefore
cannot affect the properness of a family of reals.
\end{proof}
We can push this argument further to show that it is consistent with
{\rm ZFC} that there are \emph{continuum} many proper families of
reals of size $\omega_1$.

\begin{theorem}\label{th:continuum}
There is a generic extension of $V$ by a c.c.c.\ poset that
satisfies $\neg{\rm CH}$ and contains continuum many proper families
of reals of size $\omega_1$.
\end{theorem}
\begin{proof}
We start by  forcing ${\rm MA}+\neg{\rm CH}$. Since this can be done
by a c.c.c.\ forcing notion (\cite{jech:settheory}, \hbox{p. 272}),
we can assume without loss of generality that $V\models{\rm
MA}+\neg{\rm CH}$.

Define a finite support product $\q=\Pi_{\xi<2^\omega}\p^\xi$ where
every $\p^\xi$ is an iteration of length $\omega_1$ as described in
Theorem \ref{th:forcing}. Since Martin's Axiom implies that finite
support products of c.c.c.\ posets are c.c.c.\ (see
\cite{jech:settheory}, \hbox{p. 277}), the product poset $\q$ is
c.c.c.. Let $G\subseteq \q$ be $V$-generic, then each
$G^\xi=G\upharpoonright \p^\xi$ together with $\p^\xi$ can be used
to build an arithmetically closed family $\x^\xi$ as described in
Theorem \ref{th:forcing}.  Each such $\x^\xi$ will be the union of
an increasing chain of countable arithmetically closed families
$\x^\xi_\gamma$ for $\gamma<\omega_1$. First, I claim that all
$\x^\xi$ are distinct. Fixing $\alpha<\beta$, I will show that
$\x^\alpha\neq\x^\beta$. Consider $V[G\upharpoonright
\beta+1]=V[G\upharpoonright \beta][G^\beta]$ a generic extension by
$(\q\upharpoonright\beta)\times \p^\beta$. Observe that $\x^\alpha$
already exists in $V[G\upharpoonright\beta]$.  Recall that to build
$\x^\beta$, we start with an arithmetically closed countable family
$\x_0^\beta$ and let the first poset in the iteration $\p^\beta$ be
$\x_0^\beta/\fin$. Let $g$ be the generic filter for
$\x_0^\beta/\fin$ definable from $G^\beta$. The next step in
constructing $\x^\beta$ is to pick $A\subseteq\n$ such that
$A\subseteq_\fin B$ for all $B\in g$ and define $\x_1^\beta$ to be
the arithmetic closure of $\x_0^\beta$ and $A$. It should be clear
that $g$ is definable from $A$ and $\x_0^\beta$. Since $g$ is
$V[G\upharpoonright\beta]$-generic, it follows that $g\notin
V[G\upharpoonright\beta]$.  Thus, $A\notin
V[G\upharpoonright\beta]$, and hence $\x^\beta\neq\x^\alpha$. It
remains to show that each $\x^\alpha$ is proper in $V[G]$.  Fix
$\alpha<2^\omega$ and let
$V[G]=V[G\upharpoonright\alpha][G^\alpha][G_\text{tail}]$ where
$G_\text{tail}$ is the generic for $\q$ above $\alpha$.  By the
commutativity of products,
$V[G\upharpoonright\alpha][G^\alpha][G_\text{tail}]=
V[G\upharpoonright\alpha][G_\text{tail}][G^\alpha]$ and $G^\alpha$
is $V[G\upharpoonright\alpha][G_\text{tail}]$-generic.  Fix a
countable $M\prec H_\lambda^{V[G]}$ containing the sequence $\la
\x_\xi^\alpha\mid \xi<\omega_1\ra$ as an element. By Lemma
\ref{le:intersect}, $M\cap \x^\alpha$ is some $\x_\gamma^\alpha$.
This is the key step of the proof since it allows us to know exactly
what $M\cap \x^\alpha$ is, even though we know nothing about $M$.
Let $G_\xi^\alpha=G^\alpha\upharpoonright \p_\xi^\alpha$ for
$\xi<\omega_1$. Let $\mathcal D=\{\mathscr D\cap M\mid \mathscr D\in
M\text{ and }\mathscr D \text{ dense in }\x^\alpha/\fin\}$.  There
must be some $\beta<\omega_1$ such that $\mathcal D\in
V[G\upharpoonright\alpha][G_\text{tail}][G_\beta^\alpha]$.  By
construction, there must be some stage $\delta>\beta$ at which we
forced with $\x_\gamma^\alpha/\fin$ and added a set $A$ such that
$A\subseteq_\fin B$ for all $B\in H$ where
$G_{\delta+1}^\alpha=G_\delta^\alpha
*H$.  Now observe that $H$ is $V[G\upharpoonright\alpha][G_\text{tail}]
[G^\alpha_\delta]$-generic for $\x_\gamma^\alpha/\fin$. Therefore
$H$ meets all the sets in $\mathcal D$.  So we can conclude that $A$
is $M$-generic.

A standard nice name counting argument will again show that
$(2^\omega)^V=(2^\omega)^{V[G]}$.  Thus, $V[G]$ satisfies $\neg{\rm
CH}$ and contains continuum many proper families of reals of size
$\omega_1$.
\end{proof}

Similar techniques allow us to force the existence of a piecewise
proper family of reals of size $\omega_2$.

\begin{lemma}\label{le:piecewise}
Let $\x_0\subseteq\x_1\subseteq\cdots\subseteq
\x_\xi\subseteq\cdots$ for $\xi<\omega_1$ be a continuous chain of
countable families of reals and let $\x=\cup_{\xi<\omega_1} \x_\xi$.
Assume that for every $\xi<\omega_1$, if $B\in \x_\xi$ and $\mathcal
D$ is a countable collection of dense subsets of $\x_\xi$, there is
$A\in \x/\fin$ that is $\langle \x_\xi,\mathcal D\rangle$-generic
below $B$. Then $\x$ is proper and $\x$ remains proper after forcing
with any absolutely c.c.c.\ poset.
\end{lemma}
\begin{proof}
The proof is a straightforward modification of the proof of Theorem
\ref{th:oldreals}. Let $\p$ be an absolutely c.c.c.\ poset and
$g\subseteq \p$ be $V$-generic. We need to show that $\x$ is proper
in $V[g]$. Fix a countable $M[g]\prec H_\lambda [g]$ in $V[g]$ such
that $\la \x_\xi:\xi<\omega_1\ra, \p\in M[g]$ and $M\subseteq V$.
Let $\x_\alpha=M\cap \x$ and let $\mathcal D=\{\mathscr D\cap
\x_\xi\mid\mathscr D\in M\text{ and }\mathscr D\text{ dense in
}\x\}$. Observe that $\mathcal D\subseteq V$ and $|\mathcal
D|=\omega$. Define $\mathcal E$ as in proof of Theorem
\ref{th:oldreals}. Now choose $A\in \x$ that is $\la
\x_\alpha,\mathcal E\ra$-generic in $V$. It follows that $A$ is
$M$-generic. Next proceed exactly as in the proof of Theorem
\ref{th:oldreals}, using the fact that $\p$ is absolutely c.c.c.\ in
the final stage of the argument.
\end{proof}
\begin{theorem}\label{th:piecewise}
There is a generic extension of $V$ by a c.c.c.\ poset which
contains a piecewise proper family of reals of size $\omega_2$.
\end{theorem}
\begin{proof}
We will define a c.c.c.\ forcing iteration $\q_{\omega_2}$ of length
$\omega_2$ to accomplish this. Let $\q_0$ be the forcing to add a
proper family $\x_0$ of size $\omega_1$ (Theorem \ref{th:forcing}).
At the $\alpha^{\text{th}}$-stage, force with the poset to add a
proper family $\x_{\alpha}\supseteq \cup_{\beta<\alpha}\x_\beta$.
Observe here, that the poset from Theorem \ref{th:forcing} can be
very easily modified to the poset which adds a proper family
extending any family of reals from the ground model. Let $G\subseteq
\q_{\omega_2}$ be $V$-generic. I claim each $\x_\alpha$ remains
proper in $V[G]$. Fix $\x_\alpha$ and factor the forcing
$\q_{\omega_2}=\q_\alpha*\q_\tail$. The family $\x_\alpha$ is proper
in $V[G_\alpha]$ and $\q_\tail$ is absolutely c.c.c.\ in
$V[G_\alpha]$. The poset $\q_\tail$ is absolutely c.c.c\ since the
forcing to add a proper family is a finite support iteration of
countable posets. Thus, by Lemma \ref{le:piecewise}, $\x_\alpha$
remains proper in $V[G_\alpha][G_\tail]$. Thus,
$\x=\cup_{\alpha<\omega_2} \x_\alpha$ is clearly piecewise proper.
\end{proof}
By exactly following the proof of Theorem \ref{th:continuum}, we can
extend Theorem \ref{th:piecewise} to obtain:
\begin{theorem}
There is a generic extension of $V$ by a c.c.c.\ poset which
contains continuum many piecewise proper families of reals of size
$\omega_2$.
\end{theorem}

By Enayat's \cite{enayat:endextensions} example of a non-proper
arithmetically closed family of size $\omega_1$, we know that there
are piecewise proper families that are not proper. This follows by
recalling that arithmetically closed families of size $\omega_1$ are
trivially piecewise proper. It is not clear whether every proper
family has to be piecewise proper. In particular, it is not known
whether $\power(\n)$ is piecewise proper. It follows that
$\power(\n)$ \emph{can be} piecewise proper from the proof of
Theorem \ref{th:piecewise} since we can modify the construction to
end up with $\x=\power(\n)$.

Finally, I will discuss a possible construction for proper families
under {\rm PFA}. The idea is, in some sense, to mimic the forcing
iteration like that of Theorem \ref{th:forcing} in the ground model.
Unfortunately, the main problem with the construction is that it is
not clear whether we are getting the whole $\power(\n)$. This
problem never arose in the forcing construction since we were
building families of size $\omega_1$ and knew that the continuum was
larger than $\omega_1$.  I will describe the construction and a
possible way of ensuring that the resulting family is not
$\power(\n)$.

Fix an enumeration $\{\langle A_\xi,B_\xi\rangle\mid \xi<\omega_2\}$
of $\power(\omega)\times \power(\omega)$. Also fix a bookkeeping
function $f$ from $\omega_2$ onto $\omega_2$ such that each element
appears cofinally in the range. I will build a family $\x$ of size
$\omega_2$ as the union of an increasing chain of arithmetically
closed families $\x_\xi$ for $\xi<\omega_2$. Start with any
arithmetically closed family $\x_0$ of size $\omega_1$. Suppose we
have constructed $\x_\beta$ for $\beta\leq\alpha$ and we need to
construct $\x_{\alpha+1}$. Consult $f(\alpha)=\gamma$ and consider
the pair $\langle A_\gamma,B_\gamma\rangle$ in the enumeration of
$\power(\omega)\times \power(\omega)$. First, suppose that
$A_\gamma$ codes a countable family $\y\subseteq \x_\alpha$ and
$B_\gamma$ codes a countable collection $\mathcal D$ of dense
subsets of $\y$. Let $G$ be some filter on $\y$ meeting all sets in
$\mathcal D$ and let $A\subseteq_\fin C$ for all $C\in G$. Define
$\x_{\alpha+1}$ to be the arithmetic closure of $\x_\alpha$ and $A$.
If the pair $\langle A_\gamma,B_\gamma\rangle $ does not code such
information, let $\x_{\alpha+1}=\x_\alpha$.  At limit stages take
unions.

I claim that $\x$ is proper.  Fix some countable $M\prec H_\lambda$
containing $\x$. Let $\y=M\cap \x$ and let $\mathcal D=\{\mathscr
D\cap M\mid \mathscr D\in M\text{ and }\mathscr D \text{ dense in
}\x/\fin \}$. There must be some $\gamma$ such that $\langle
A_\gamma, B_\gamma\rangle$ codes $\y$ and $\mathcal D$. Let $\delta$
such that $M\cap \x$ is contained in $\x_\delta$, then there must be
some $\alpha>\delta$ such that $f(\alpha)=\gamma$. Thus, at stage
$\alpha$ in the construction we considered the pair $\langle
A_\gamma,B_\gamma\rangle$. Since $\alpha>\delta$, we have $M\cap
\x=M\cap \x_\alpha$.  It follows that at stage $\alpha$ we added an
$M$-generic set $A$ to $\x$.

A way to prove that $\x\neq \power(\n)$ would be to show that some
fixed set $C$ is not in $\x$.  Suppose the following question had a
positive answer:
\begin{question}\label{con:proper}
Let $\x$ be an arithmetically closed family such that $C\notin \x$
and $\y\subseteq \x$ be a countable family. Is there a
$\y/\fin$-name $\dot{A}$ such that $1_{\y/\fin}\Vdash ``
\dot{A}\subseteq_\fin B\text{ for all }B\in \dot{G}\text{ and
}\check{C}$ is not in the arithmetic closure of $\dot{A}\text{ and
}\check{\x}$"?
\end{question}

Assuming that the answer to Question \ref{con:proper} is positive,
let us construct a proper family $\x$ in such a way that $C$ is not
in $\x$.  We will carry out the above construction being careful in
our choice of the filters $G$ and elements $A$. Start with $\x_0$
that does not contain $C$ and assume that $C\notin \x_\alpha$.
Suppose the pair $\langle A_\gamma,B_\gamma\rangle$ considered at
stage $\alpha$ codes meaningful information.  That is, $A_\gamma$
codes a countable family $\y\subseteq \x_\alpha$ and $B_\gamma$
codes a countable collection $\mathcal D$ of dense subsets of $\y$.
Choose some transitive $N\prec H_{\omega_2}$ of size $\omega_1$ such
that $\x_\alpha$, $\y$, and $\mathcal D$ are elements of $N$. Since
we assumed a positive answer to Question \ref{con:proper},
$H_{\omega_2}$ satisfies that there exists a $\y/\fin$-name
$\dot{A}$ such that $1_{\y/\fin}\Vdash ``\dot{A}\subseteq_\fin
B\text{ for all }B\in \dot{G}\text{ and }\check{C}$ is not in the
arithmetic closure of $\dot{A}\text{ and }\check{\x_\alpha}$". But
then $N$ satisfies the same statement by elementarity. Hence there
is $\dot{A}\in N$ such that $N$ satisfies $1_{\y/\fin}\Vdash
``\dot{A}\subseteq_\fin B\text{ for all }B\in \dot{G}\text{ and
}\check{C}$ is not in the arithmetic closure of $\dot{A}\text{ and}$
$\check{\x_\alpha}$". Now use {\rm PFA} to find an $N$-generic
filter $G$ for $\y/\fin$. Since $G$ is fully generic for the model
$N$, the model $N[G]$ will satisfy that $C$ is not in the arithmetic
closure of $\x_\alpha$ and $A=\dot{A}_G$. Thus, it is really true
that $C$ is not in the arithmetic closure of $\x_\alpha$ and $A$.
Since $G$ also met all the dense sets in $\mathcal D$ and
$A\subseteq_\fin B$ for all $B\in G$, we can let $\x_{\alpha+1}$ be
the arithmetic closure of $\x_\alpha$ and $A$.  Thus, $C\notin
\x_{\alpha+1}$. We can conclude that $C\notin\x$.
\section{Questions}
\begin{question}
Can {\rm ZFC} or {\rm ZFC} + {\rm PFA} prove the existence of an
uncountable proper family of reals other than $\power(\n)$?
\end{question}
\begin{question}
Can {\rm ZFC} or {\rm ZFC} + {\rm PFA} prove the existence of a
piecewise proper family of size $\omega_2$?
\end{question}
\begin{question}
Is it consistent with {\rm ZFC} that there are proper families of
reals of size $\omega_2$ other than $\power(\n)$?
\end{question}
\begin{question}
What is the answer to Question \ref{con:proper}?
\end{question}
\begin{question}
Can $\power(\n)$ be non-piecewise proper?
\end{question}
\bibliographystyle{alpha}
\bibliography{database}
\end{document}